\documentclass[reqno]{amsart}
\usepackage{amsfonts,amsmath,amssymb}
\usepackage{graphicx}
\usepackage[dvipsnames]{xcolor}

\numberwithin{equation}{section}

\newtheorem{theorem}{Theorem}[section]
\newtheorem{lemma}{Lemma}[section]

\DeclareFontFamily{U}{mathx}{\hyphenchar\font45}
\DeclareFontShape{U}{mathx}{m}{n}{
<5> <6> <7> <8> <9> <10>
<10.95> <12> <14.4> <17.28> <20.74> <24.88>
mathx10
}{}
\DeclareSymbolFont{mathx}{U}{mathx}{m}{n}
\DeclareFontSubstitution{U}{mathx}{m}{n}
\DeclareMathAccent{\widecheck}{0}{mathx}{"71}

\usepackage[margin=1.30in]{geometry}
\numberwithin{equation}{section}

\begin{document}
\title[Global endpoint Strichartz estimates on $\mathbb{R}\times\mathbb{T}$]
{Global endpoint Strichartz estimates for Schr\"odinger equations on the cylinder $\mathbb{R}\times\mathbb{T}$}

\author{Alexander Barron}
\address{University of Illinois at Urbana-Champaign}
\email{aabarron@illinois.edu}
\author{Michael Christ}
\address{University of California, Berkeley}
\email{mchrist@berkeley.edu}
\author{Benoit Pausader}
\address{Brown University}
\email{benoit\_pausader@brown.edu}

 \thanks{The second author was supported in part by NSF grant DMS-1901413. The third author was supported in part by NSF grant DMS-1700282.}

\maketitle

\section{Long-time, scaling-critical Strichartz estimates on $\mathbb{R}\times\mathbb{T}$}

Define the norm on $\mathbb{R}\times\mathbb{R}\times\mathbb{T}=(\mathbb{Z}+[0,1))\times\mathbb{R}\times \mathbb{T}$:
\begin{equation}\label{MixedTypeNorms}
\Vert u\Vert^a_{\ell^aL^b(\mathbb{R},L^c(\mathbb{R}\times\mathbb{T}))}:=\sum_{\gamma\in\mathbb{Z}}\left(\int_{s\in[0,1)}\left( \int_{x,y\in\mathbb{R}\times\mathbb{T}}\vert u(\gamma+s,x,y)\vert^c\, dxdy\right)^\frac{b}{c}ds\right)^\frac{a}{b}.
\end{equation}
In this paper, we prove the following global in time Strichartz-type estimate:
\begin{theorem}
There exists $C<\infty$ such that for all $f\in L^2(\mathbb{R}\times\mathbb{T})$,
\begin{equation}\label{Stric}
\begin{split}
\Vert e^{it\Delta_{\mathbb{R}\times\mathbb{T}}}f\Vert_{\ell^8L^4(\mathbb{R},L^4(\mathbb{R}\times\mathbb{T}))}&\le  C\Vert f\Vert_{L^2(\mathbb{R}\times\mathbb{T})}.
\end{split}
\end{equation}
\end{theorem}
This inequality is saturated\footnote{In the sense that the quotient of both sides converges to a nonzero constant as $n\to\infty$.} by two different families of functions of $(x,y)\in\mathbb{R}\times\mathbb{T}$:
\begin{equation}\label{OptimizerStric}
\begin{split}
F_n(x,y)=nG(n\sqrt{x^2+y^2})\mathfrak{1}_{\{n(x^2+y^2)\le 1\}},\quad f_n(x,y)=n^{-\frac{1}{2}}G(n^{-1}x),
\end{split}
\end{equation}
where $G(s)=e^{-s^2}$ is a Gaussian. These correspond respectively to saturators for Strichartz estimates in $2d$ and in $1d$ \cite{St}. The exponents in \eqref{Stric} are optimal in the following sense: $(i)$ on the one hand, since $e^{it\Delta_{\mathbb{R}\times\mathbb{T}}}f_n(x,y)=n^{-\frac{1}{2}}(e^{in^{-2}t\partial_{xx}}G)(n^{-1}x)$ behaves as a (low-frequency) solution of the Schr\"odinger equation on $\mathbb{R}$, the exponent $8$ in \eqref{Stric} cannot be lowered; $(ii)$ on the other hand, since $e^{it\Delta_{\mathbb{R}\times\mathbb{T}}}F_n$ behaves as a (high-frequency) solution of the Schr\"odinger equation in $\mathbb{R}^2$ (see e.g. \cite[Lemma 4.2]{IoPa} for similar computations), the exponent $4$ cannot be changed if the righthand side is measured in $L^2$.

Interpolating with the estimate\footnote{This follows from variants of classical $TT^\ast$ estimates as in Ginibre-Velo \cite{GinVel}, see \cite[Section 3]{HaPa}.}
when $q=4$ and $p=\infty$,
\begin{equation*}
\begin{split}
\Vert e^{it\Delta_{\mathbb{R}\times\mathbb{T}}}P_{\le N}f\Vert_{\ell^4L^\infty(\mathbb{R},L^\infty(\mathbb{R}\times\mathbb{T}))}&\lesssim N\Vert f\Vert_{L^2(\mathbb{R}\times\mathbb{T})},\qquad \mathcal{F}\left\{P_{\le N}f\right\}(\xi,k)=\varphi(\xi/N)\varphi(k/N)\widehat{f}(\xi,k),
\end{split}
\end{equation*}
where $\varphi\in C^\infty_c(\mathbb{R})$ is a smooth bump function, and using boundedness of the square function, 
we obtain the family of scaling invariant Strichartz estimates on $\mathbb{R}\times\mathbb{T}$:
\begin{equation}\label{Stric2}
\begin{split}
\Vert e^{it\Delta_{\mathbb{R}\times\mathbb{T}}}f\Vert_{\ell^qL^p(\mathbb{R},L^p(\mathbb{R}\times\mathbb{T}))}&\lesssim \Vert f\Vert_{H^s(\mathbb{R}\times\mathbb{T})},\quad
\frac{2}{q}+\frac{1}{p}=\frac{1}{2},\quad 4< q\le8,\quad s=1-\frac{4}{p}.
\end{split}
\end{equation}

%\textcolor{blue}{Interpolating with the simple estimate\footnote{This follows from classical $TT^\ast$ estimates as in Ginibre-Velo \cite{GinVel}.}
%when $q=4$ and $p=\infty$, 
%%\begin{equation*}
%%\begin{split}
%%\Vert e^{it\Delta_{\mathbb{R}\times\mathbb{T}}}P_{\le N}f\Vert_{\ell^4L^\infty(\mathbb{R},L^\infty(\mathbb{R}\times\mathbb{T}))}&\lesssim N\Vert f\Vert_{L^2(\mathbb{R}\times\mathbb{T})},\qquad \mathcal{F}\left\{P_{\le N}f\right\}(\xi,k)=\mathfrak{1}_{\{-N\le \vert\xi\vert,\vert k\vert\le N\}}\widehat{f}(\xi,k)
%%\end{split}
%%\end{equation*}
%%and using boundedness of the square function, 
%we obtain the family of scaling invariant Strichartz estimates on $\mathbb{R}\times\mathbb{T}$:
%\begin{equation}\label{Stric2}
%\begin{split}
%\Vert e^{it\Delta_{\mathbb{R}\times\mathbb{T}}}f\Vert_{\ell^qL^4(\mathbb{R},L^p(\mathbb{R}\times\mathbb{T}))}&\lesssim \Vert f\Vert_{H^s(\mathbb{R}\times\mathbb{T})},\quad
%\frac{2}{q}+\frac{1}{p}=\frac{1}{2},\quad 4\le  q\le8,\quad s=1-\frac{4}{p}.
%\end{split}
%\end{equation}}

Strichartz-type inequalities with mixed norms in the time variable of the form \eqref{MixedTypeNorms} were introduced in \cite{HaPa} to study the asymptotic behavior of solutions to critical NLS on product spaces $\mathbb{R}^n\times\mathbb{T}^d$ which are examples of manifolds where the global dimension is smaller than the local dimension. Similar cases were later explored in \cite{ChGuZh,Zh1,Zh2} and the sharp results when $s>0$ was obtained in \cite{Ba} using results from $\ell^2$-decoupling \cite{BoDe}.

However, to study NLS with data in $L^2$, estimates with loss of derivatives are useless. This raised the question of whether a Strichartz-type inequality with no loss of derivatives could hold for Schr\"odinger equations on $d$-dimensional manifolds smaller at infinity than $\mathbb{R}^d$. For the torus $\mathbb{T}^d$, for instance, a lossless inequality like \eqref{Stric} does not hold, not even locally in time (that is, with $a=\infty$) as observed in \cite{Bo}. In fact, for manifolds ``smaller'' than $\mathbb{R}^2$, the only estimate known to the authors is the result from \cite{TaTz} which obtains local version of \eqref{Stric} (with $a=\infty$ instead of $a=8$). We refer e.g. to \cite{Bouc,BuGuHa,Chr} for the study of Strichartz estimates without losses in the presence of trapped geodesic.

As for nonlinear applications of \eqref{Stric}, one can easily show local well-posedness of the cubic NLS in $L^2(\mathbb{R}\times\mathbb{T})$, recovering the result in \cite{TaTz}. However, the long-time behavior is modified scattering as shown in \cite{HaPaTzVi}, which requires more information (and stronger control on initial data) than $L^2$-Strichartz estimates and it remains a challenging open question as to whether nonlinear solutions satisfy global bounds of the type \eqref{Stric}.

This leaves open some interesting questions:
\begin{enumerate}

\item Can one extend this result to other semi-periodic settings, i.e., does an estimate like
\begin{equation*}
\begin{split}
\Vert e^{it\Delta_{\mathbb{R}^d\times\mathbb{T}^n}}f\Vert_{\ell^qL^p(\mathbb{R},L^p(\mathbb{R}^d\times\mathbb{T}^n))}\lesssim \Vert f\Vert_{L^2(\mathbb{R}^d\times\mathbb{T}^n)},\quad p=\frac{2(n+d+2)}{n+d},\,\, q=\frac{2(n+d+2)}{d}.
\end{split}
\end{equation*}
hold? This is settled for $n+d\le 2$, but for higher values, $p<4$ and the problem is much more challenging.

\item Can one understand and characterize optimizers of \eqref{Stric}? In principle, introducing a parameter for the length of the torus (or the local time interval), one may expect that optimizers should vary smoothly between the two families in \eqref{OptimizerStric}.

\item Can one obtain a good profile decomposition, i.e., study the defect of compactness of bounded sequences in $L^2(\mathbb{R}\times\mathbb{T})$?

\end{enumerate}

\section{Proof of Theorem \ref{Stric}}

Since the analysis is done purely in the frequency space, we pass to the Fourier transform and consider $f\in L^2(\mathbb{R}\times\mathbb{Z})$, which corresponds to the Fourier transform of the function in \eqref{Stric}. By homogeneity, we may choose $f$ to be of unit $L^2$ norm and by density we may assume that $f$ is compactly supported so that all integrals below converge absolutely. We let $\mathbb{T}=\mathbb{R}/2\pi\mathbb{Z}$ and we define the Fourier transform on $\mathbb{R}\times\mathbb{Z}$
\begin{equation*}
\begin{split}
\widehat{f}(x,y)=\sum_{k\in\mathbb{Z}}\int_{\mathbb{R}}f(\xi,k)e^{ix\xi}e^{iky}dx,\qquad\widecheck{g}(\xi,k)=\frac{1}{(2\pi)^2}\int_{\mathbb{R}}\int_{y=0}^{2\pi} g(x,y)e^{-ix\xi}e^{-iky}dydx.
\end{split}
\end{equation*}
Since we will take Fourier transforms, it will be convenient to replace the integral over $[0,1)$ in \eqref{MixedTypeNorms} by an integral over $\mathbb{R}$. To do this, we introduce a Gaussian cutoff in time and let
\begin{equation}\label{DefJGamma}
J_\gamma:=\Vert e^{-\frac{1}{4}(t-\gamma)^2}e^{it\Delta_{\mathbb{R}\times\mathbb{T}}}\widehat{f}\Vert_{L^4_{x,y,t}(\mathbb{R}\times\mathbb{T}\times\mathbb{R})}.
\end{equation}
%The Gaussian weight in time in the definition of $J_\gamma$ allows to replace integration over a finite interval of time by a global-in-time integration and leads to nicer formulas for Fourier transforms in time.
To prove \eqref{Stric}, it will suffice to control the $\ell^8$-norm of $J_\gamma$. For simplicity of presentation, we let
\begin{equation*}
\begin{split}
{\vec \xi}&=(\xi_1,\xi_2,\xi_3,\xi_4),\qquad\quad {\vec k}=(k_1,k_2,k_3,k_4),\\
\langle \xi\rangle&=\xi_1-\xi_2+\xi_3-\xi_4=\langle  {\vec \xi},(1,-1,1,-1)\rangle,\qquad\langle { k}\rangle=k_1-k_2+k_3-k_4,\\
f_j&=f(\xi_j,k_j),\,\, j\in\{1,3\},\qquad f_j=\overline{f}(\xi_j,k_j),\,\, j\in\{2,4\},\\
Q( \xi, k)&=\vert \xi_1\vert^2+\vert \xi_3\vert^2-\vert \xi_2\vert^2-\vert \xi_4\vert^2+\vert k_1\vert^2+\vert k_3\vert^2-\vert k_2\vert^2-\vert k_4\vert^2.
\end{split}
\end{equation*}
We substitute $t\to t+\gamma$ in \eqref{DefJGamma} and expand $J_\gamma^4$ into
\begin{equation*}
\begin{split}
J_\gamma^4&=\int_{x,y,t}\left[\sum_{k_1\dots k_4} \int_{\xi_1\dots\xi_4}\Pi_{j=1}^4 f_j\cdot e^{- t^2}e^{-i(t+\gamma)Q(\xi, k)}\cdot e^{i x\langle \xi\rangle}e^{iy\langle  k\rangle} d{\vec\xi}\right] dxdydt\\
&=4\pi^\frac{5}{2}\sum_{k_1\dots k_4} \int_{\xi_1\dots\xi_4}\Pi_j f_j\cdot e^{-\frac{1}{4} (Q(\xi, k))^2}e^{-i\gamma Q(\xi,k)}\cdot \delta(\langle \xi\rangle)\delta(\langle k\rangle) d{\vec \xi}.
\end{split}
\end{equation*}
An argument of Takaoka-Tzevtkov \cite{TaTz} shows that each individual $J_{\gamma}^{4}$ is bounded, but we need to handle the sum in $\gamma$. We square $J_{\gamma}^{4}$ and sum over $\gamma$ to get
\begin{equation}\label{DefN}
\begin{split}
\mathcal{J}&:=\sum_{\gamma\in\mathbb{Z}}J_\gamma^8\\
&=16\pi^5\sum_{\substack{k_1\dots k_4\\ k^\prime_1\dots k^\prime_4}} \int_{\substack{\xi_1\dots\xi_4,\\\xi_1^\prime\dots\xi_4^\prime}}\Pi_{j=1}^4 f_j\overline{\Pi_{l=1}^4f_l^\prime}\cdot  e^{-\frac{1}{4} (Q(\xi,k))^2}e^{-\frac{1}{4} (Q(\xi^\prime, k^\prime))^2}\cdot \sum_\gamma e^{-i\gamma \left[ Q(\xi, k)-Q(\xi^\prime, k^\prime)\right]}\\
&\qquad\cdot \delta(\langle \xi\rangle) \delta(\langle \xi^\prime\rangle)\delta(\langle k\rangle)\delta(\langle k^\prime\rangle) d{\vec \xi}d{\vec \xi^\prime}.
\end{split}
\end{equation}
Using Poisson summation in $\gamma$ we observe that
\begin{equation*}
\begin{split}
\sum_{\gamma\in\mathbb{Z}} e^{-i\gamma \left[ Q( \xi,k)-Q(\xi^\prime, k^\prime)\right]}&=2\pi \sum_{\mu\in 2\pi \mathbb{Z}} \delta(\mu-Q(\xi, k)+Q(\xi^\prime, k^\prime)).
\end{split}
\end{equation*}
Introducing the new notations
\begin{equation*}
\begin{split}
\Xi&:=(\xi_1,\xi_3,\xi_2^\prime,\xi_4^\prime),\qquad\Xi^\prime:=(\xi_2,\xi_4,\xi_1^\prime,\xi_3^\prime),\\
K&:=(k_1,k_3,k_2^\prime,k_4^\prime),\qquad K^\prime:=(k_2,k_4,k_1^\prime,k_3^\prime),\\
F(\Xi,K)&:=f(\xi_1,k_1)f(\xi_3,k_3)f(\xi_2^\prime,k_2^\prime)f(\xi_4^\prime,k_4^\prime),\quad F(\Xi^\prime,K^\prime):=f(\xi_2,k_2)f(\xi_4,k_4)f(\xi_1^\prime,k_1^\prime)f(\xi_3^\prime,k_3^\prime),\\
\phi_\mu&:=\mu-Q(\xi, k)+Q(\xi^\prime, k^\prime)=\mu-\vert\Xi\vert^2-\vert K\vert^2+\vert\Xi^\prime\vert^2+\vert K^\prime\vert^2,
\end{split}
\end{equation*}
we arrive at
\begin{equation*}
\begin{split}
\mathcal{J}
&=32\pi^6 \sum_{K,K^\prime\in\mathbb{Z}^4} \int_{\Xi,\Xi^\prime}F(\Xi,K)\overline{F(\Xi^\prime,K^\prime)} \cdot \mathcal{K}(\Xi,K;\Xi^\prime,K^\prime)\cdot d\Xi d\Xi^\prime\\
\mathcal{K}(\Xi,K;\Xi^\prime,K^\prime)&:= e^{-\frac{1}{4} \left[ Q(\xi, k)^2+Q(\xi^\prime, k^\prime)^2\right]}\cdot \sum_{\mu\in2\pi \mathbb{Z}}\delta(\phi_\mu)\delta(\langle \xi\rangle) \delta(\langle\xi^\prime\rangle)\delta(\langle k\rangle)\delta(\langle k^\prime\rangle).\\
\end{split}
\end{equation*}
Using the Schur test, the inequality \eqref{Stric} follows from the next lemma.

\begin{lemma}\label{ShurLem}

With the notations above,
\begin{equation*}
\begin{split}
\sup_{(\Xi,K)\in\mathbb{R}^4\times\mathbb{Z}^4}\sum_{K^\prime\in\mathbb{Z}^4}\int\mathcal{K}(\Xi,K;\Xi^\prime, K^\prime)d\Xi^\prime<\infty.
\end{split}
\end{equation*}

\end{lemma}

\begin{proof}[Proof of Lemma \ref{ShurLem}]
We need to bound
\begin{equation}\label{SumShur1}
\begin{split}
\sum_{\mu\in 2\pi\mathbb{Z}}\sum_{K^\prime\in\mathbb{Z}^4}\int_{\Xi^\prime\in\mathbb{R}^4}&e^{-\frac{1}{4} \left[ Q(\xi, k)^2+Q(\xi^\prime, k^\prime)^2\right]}\delta(\phi_\mu)\delta(\langle\xi\rangle) \delta(\langle \xi^\prime\rangle)\delta(\langle k\rangle)\delta(\langle k^\prime\rangle) d\Xi^\prime
\end{split}
\end{equation}
uniformly in $(\Xi,K)\in\mathbb{R}^4\times\mathbb{Z}^4$. Below we occasionally write $Q = Q(\xi, k)$ and $Q' = Q(\xi', k').$ 

\medskip

Using the polarization identity on the support of $\delta(\mu-Q+Q^\prime)$, we can bound
\begin{equation*}
\begin{split}
e^{-\frac{1}{4} \left[ Q^2+(Q^\prime)^2\right]}&=e^{-\frac{1}{8} \left[ Q^2+(Q^\prime)^2\right]}e^{-\frac{1}{16} \left[ (Q+Q^\prime)^2+(Q-Q^\prime)^2\right]}\le e^{-\frac{1}{16}\mu^2}e^{-\frac{1}{8} \left[ Q^2+(Q^\prime)^2\right]}.
\end{split}
\end{equation*}
%and we can easily sum over $\mu$ in \eqref{SumShur1}. 
Moreover, when $\langle \xi\rangle=0=\langle k\rangle$ we can substitute $$\xi_4 = \xi_1 - \xi_2 + \xi_3 \ \  \text{ and } \ \ k_4 = k_1 - k_2 + k_3 $$ into $Q$ and then factor to obtain 
\begin{equation*}
\begin{split}
Q( \xi, k)&=-2\left[\vert (\xi_2-c_x,k_2-c_y)\vert^2-R^2\right],\\
(c_x,c_y)&=(\frac{\xi_1+\xi_3}{2},\frac{k_1+k_3}{2}),\qquad R^2=\left(\frac{\xi_1-\xi_3}{2}\right)^2+\left(\frac{k_1-k_3}{2}\right)^2.
\end{split}
\end{equation*}
A similar identity holds for $Q'$ when $\langle \xi ' \rangle = 0 = \langle k' \rangle$. Indeed, on the support of $\delta(\langle \xi' \rangle)\delta(\langle k' \rangle)$ we can substitute $$ \xi_{3}' = -\xi_{1}' + \xi_{2}' + \xi_{4}' \ \ \text{ and } \ \ k_{3}' = -k_{1}' + k_{2}' + k_{4}' $$ into $Q'$ and factor to obtain 
\begin{equation*}
\begin{split}
Q( \xi', k')&=2\left[\vert (\xi_{1}' -c'_x,k_{1}' -c'_y)\vert^2-(R')^2\right],\\
(c_{x}',c'_y)&=(\frac{\xi'_2+\xi'_4}{2},\frac{k'_2 + k'_4}{2}),\qquad (R')^2=\left(\frac{\xi'_2-\xi'_4}{2}\right)^2+\left(\frac{k'_2-k'_4}{2}\right)^2.
\end{split}
\end{equation*}
With these substitutions made, notice that
\begin{equation*}
\begin{split}
\phi_{\mu} = \mu + 2[ |(\xi_2 - c_x, k_2 - c_y)|^2 - R^2] + 2[ |(\xi'_1 - c'_x, k'_1 - c'_y )|^2 - (R')^2 ]
\end{split}
\end{equation*}
and therefore
\begin{equation*}
\begin{split}
\delta(\phi_{\mu}) = \frac{1}{2}\delta( |(\xi_2 - c_x, k_2 - c_y)|^2 +  |(\xi'_1 - c'_x, k'_1 - c'_y )|^2 - A_{\mu}  ), \qquad A_{\mu} = \frac{ R^2 + (R')^2 - \mu}{2}.
\end{split}
\end{equation*}
Using these observations to estimate \eqref{SumShur1} we arrive at
\begin{align*} \eqref{SumShur1} \leq \frac{1}{2}\sum_{\mu \in 2\pi \mathbb{Z} } e^{-\frac{1}{16} \mu^2} \sum_{k_2, k'_1} \int_{\mathbb{R}^2}&e^{-\frac{1}{2} \left[ \left[\vert (\xi_{2} -c_x,k_{2} -c_y)\vert^2-R^2\right]^2+\left[\vert (\xi_{1}' -c'_x,k_{1}' -c'_y)\vert^2-(R')^2\right]^2\right]} \\ &\delta( |(\xi_2 - c_x, k_2 - c_y)|^2 +  |(\xi'_1 - c'_x, k'_1 - c'_y )|^2 - A_{\mu}  )  d\xi_2 d\xi'_1
\end{align*} with $c_x, c_y, c'_x, c'_y,$ and $A_{\mu}$ defined as above. Notice that $R$ and $R^\prime$ only depend on $(\Xi,K)$, and these variables have been fixed. Since we also have exponential decay in $\mu$ it therefore suffices to bound the integral
\begin{equation}\label{DefI}
\begin{split}
{\bf I} :=\sum_{\kappa,\kappa^\prime}\int_{\zeta,\zeta^\prime}&e^{-\frac{1}{2}\left[ \vert\vert(\zeta,\kappa)-{\vec C}\vert^2-R^2\vert^2+\vert \vert(\zeta^\prime,\kappa^\prime)-{\vec C}^\prime\vert^2-(R^\prime)^2\vert^2\right]}\delta(\vert (\zeta,\kappa)-{\vec C}\vert^2+\vert (\zeta^\prime,\kappa^\prime)-{\vec C^\prime}\vert^2-A)d\zeta d\zeta^\prime
\end{split}
\end{equation}
uniformly in ${\vec C},{\vec C}^\prime\in\mathbb{R}^2$, $A,R,R^\prime\in\mathbb{R}$. Moreover, since $2c_y$ and $2c'_y$ are both integers we can assume the second components of $\vec{C}, \vec{C'}$ are in $\frac{1}{2}\mathbb{Z}$.  

\medskip

The integral in \eqref{DefI} is invariant with respect to translation on $(\mathbb{R} \times \mathbb{Z}) \times (\mathbb{R} \times \mathbb{Z})$, and we may therefore assume that ${\vec C}=(0,c)$, ${\vec C}^\prime=(0,c^\prime)$ for $c,c^\prime\in\{0, \frac{1}{2} \}$. To control ${\bf I}$ we introduce sets where the exponential factors behave nicely. When $R\ge 50$, we let
\begin{equation}\label{DefSets}
\begin{split}
\mathcal{S}_0&:=\{\vert(\zeta,\kappa)-{\vec C}\vert-R\vert\le R^{-1}\},\\
\mathcal{S}_j&:=\{\vert(\zeta,\kappa)-{\vec C}\vert-R\vert\in R^{-1}[j,j+1]\},\qquad 1\le j \le R^{\frac{1}{2}}+1,\\
\mathcal{S}_\infty&:=\{\vert (\zeta,\kappa)-{\vec C}\vert -R\vert\ge R^{-\frac{1}{2}}\}
\end{split}
\end{equation}
and when $R\le 50$, we let $S_j = \emptyset$ and $\mathcal{S}_\infty=\mathbb{R}\times\mathbb{Z}$. These satisfy
\begin{equation}\label{ControlExp}
\begin{split}
\mathfrak{1}_{\mathcal{S}_j}(\zeta,\kappa)e^{-\frac{1}{2}\left[ \vert\vert(\zeta,\kappa) - \vec{C} \vert^2-R^2\vert^2\right]}&\lesssim e^{-\frac{1}{2} j^2}\mathfrak{1}_{\mathcal{S}_j}(\zeta,\kappa),\quad 0\le j\le R^\frac{1}{2}+1,\\
\mathfrak{1}_{\mathcal{S}_\infty}(\zeta,\kappa)e^{-\frac{1}{2}\left[ \vert\vert(\zeta,\kappa) - \vec{C} \vert^2-R^2\vert^2\right]}&\lesssim e^{-\frac{1}{2} \vert(\zeta,\kappa) - \vec{C} \vert}\mathfrak{1}_{\mathcal{S}_\infty}(\zeta,\kappa).
\end{split} 
\end{equation}
Indeed, the estimate on $\mathcal{S}_{j}$ in \eqref{ControlExp} follows by factoring the term in the exponential. To prove the estimate on $\mathcal{S}_{\infty}$ note that if $(\zeta, \kappa) \in \mathcal{S_{\infty}}$ and $R \geq 50$ then
\begin{equation*}
\begin{split}
\vert\vert(\zeta,\kappa) - \vec{C} \vert^2-R^2\vert^2 \geq \big[R^{-\frac{1}{2}}(|(\zeta, \kappa) - \vec{C}| + R)  \big]^2\geq |(\zeta, \kappa) - \vec{C}| + R.
\end{split}
\end{equation*}
On the other hand
\begin{equation*}
\begin{split}
\vert\vert(\zeta,\kappa) - \vec{C} \vert^2-R^2\vert^2 \geq |(\zeta,\kappa) - \vec{C} |-R-2,
\end{split}
\end{equation*}
and the estimate in \eqref{ControlExp} in $\mathcal{S}_{\infty}$ follows if $R\le 50$.
%On the other hand if  $R\le 50$ and $|(\zeta, \kappa) - \vec{C}| \geq R+1$ then
%\begin{equation*}
%\begin{split}
%\vert\vert(\zeta,\kappa) - \vec{C} \vert^2-R^2\vert^2 \geq |(\zeta,\kappa) - \vec{C} |,
%\end{split}
%\end{equation*}
%and the estimate in \eqref{ControlExp} in $\mathcal{S}_{\infty}$ follows. Finally if $R+|(\zeta, \kappa) - \vec{C}| \leq 100$ the estimate is trivial.  

\medskip

We first use \eqref{SimpleBd1} from Lemma \ref{simpleLemma} to control the contribution of $\mathcal{S}_{\infty}$ to \eqref{DefI}. In particular  
\begin{equation*}
\begin{split}
{\bf I}_{\infty\infty}& :=\sum_{\kappa,\kappa^\prime}\iint \mathfrak{1}_{\mathcal{S}_\infty}(\zeta,\kappa)\mathfrak{1}_{\mathcal{S}_\infty}(\zeta^\prime,\kappa^\prime)e^{-\frac{1}{2}\left[ \vert\vert(\zeta,\kappa - c)\vert^2-R^2\vert^2+\vert \vert(\zeta^\prime,\kappa^\prime - c')\vert^2-(R^\prime)^2\vert^2\right]}\\
&\qquad\qquad\,\cdot \delta(\vert\zeta\vert^2+\vert\kappa-c\vert^2+\vert\zeta^\prime\vert^2+\vert\kappa^\prime-c^\prime\vert^2-A)d\zeta d\zeta^\prime\\
&\,\lesssim\sum_{\kappa,\kappa^\prime}e^{-\frac{1}{2}(\vert\kappa - c\vert+\vert\kappa^\prime - c'\vert)}\iint \delta(\vert \zeta\vert^2+\vert \kappa - c\vert^2+\vert \zeta^\prime\vert^2+\vert\kappa^\prime - c^\prime\vert^2-A)d\zeta d\zeta^\prime\\
&\,\lesssim \sup_{B\in\mathbb{R}}\iint \delta( \vert\zeta\vert^2+\vert\zeta^\prime\vert^2-B)d\zeta d\zeta^\prime\lesssim 1.
\end{split}
\end{equation*}

\medskip

Next, we consider
\begin{equation*}
\begin{split}
{\bf I}_{j\infty}&=\sum_{\kappa,\kappa^\prime}\iint \mathfrak{1}_{\mathcal{S}_j}(\zeta,\kappa)\mathfrak{1}_{\mathcal{S}_\infty}(\zeta^\prime,\kappa^\prime)e^{-\frac{1}{2}\left[ \vert\vert(\zeta,\kappa)\vert^2-R^2\vert^2+\vert \vert(\zeta^\prime,\kappa^\prime)\vert^2-(R^\prime)^2\vert^2\right]}\\
&\qquad\qquad\cdot \delta(\vert \zeta\vert^2+\vert\kappa - c\vert^2+\vert \zeta^\prime\vert^2+\vert\kappa^\prime - c^\prime\vert^2-A)d\zeta d\zeta^\prime\\
&\lesssim e^{-\frac{1}{2}j^2}\sum_{\kappa^\prime}e^{-\frac{1}{2}\vert\kappa^\prime - c'\vert}\sup_B \sum_{\kappa}\iint e^{-\frac{1}{2}\vert\zeta^\prime\vert}\mathfrak{1}_{\mathcal{S}_j}(\zeta,\kappa)\delta(\vert \zeta\vert^2+\vert \kappa - c\vert^2+\vert \zeta^\prime\vert^2-B) d\zeta d\zeta^\prime.\\
\end{split}
\end{equation*}
We can split the integral above into two regions: $(i)$ when $\vert \kappa\vert\in [R-10,R+2]$, the sum is only over a uniformly bounded number of $\kappa$ and we can use \eqref{SimpleBd1}; and $(ii)$ when $\vert \kappa\vert\le R-10$, in which case we use \eqref{SimpleBd2} and the rapid decay of $e^{-\vert\zeta^\prime\vert}$. In both cases, we obtain a bounded contribution after summing over $j$.

\medskip

Finally, by symmetry, it remains to consider:
\begin{equation*}
\begin{split}
{\bf I}_{jp}&=\sum_{\kappa,\kappa^\prime}\iint \mathfrak{1}_{\mathcal{S}_j}(\zeta,\kappa)\mathfrak{1}_{\mathcal{S}_p}(\zeta^\prime,\kappa^\prime)\mathfrak{1}_{\{\vert\zeta^\prime\vert\le\vert\zeta\vert\}}e^{-\frac{1}{2}\left[ \vert\vert(\zeta,\kappa - c)\vert^2-R^2\vert^2+\vert \vert(\zeta^\prime,\kappa^\prime - c')\vert^2-(R^\prime)^2\vert^2\right]}\\
&\qquad\qquad \cdot \delta(\vert (\zeta,\kappa - c,\zeta^\prime,\kappa^\prime- c')\vert^2-A)d\zeta d\zeta^\prime.\\
\end{split}
\end{equation*}
Note that we may assume $R,R^\prime\ge50$ since otherwise $\mathcal{S}_{j}$ or $\mathcal{S}_{p}$ is empty. Using \eqref{ControlExp} we estimate
\begin{equation*}
\begin{split}
{\bf I}_{jp}&\le 2e^{-\frac{1}{2} (j^2+p^2)}\left[{\bf J}_{jp}^1+{\bf J}_{jp}^2\right],\\
{\bf J}^1_{jp}&=\sum_{R-10\le \vert \kappa\vert,\vert\kappa^\prime\vert\le R+10}\iint\mathfrak{1}_{\mathcal{S}_j}\mathfrak{1}_{\mathcal{S}_p}\delta(\vert \zeta\vert^2+\vert\kappa - c\vert^2+\vert\zeta^\prime\vert^2+\vert\kappa^\prime - c^\prime\vert^2-A)d\zeta d\zeta^\prime,\\
{\bf J}^2_{jp}&=\sum_{\kappa^\prime}\int_{\zeta^\prime}\mathfrak{1}_{\mathcal{S}_p}\left( \sum_{\vert \kappa\vert\le R-10}\int_\zeta \mathfrak{1}_{\mathcal{S}_j}\delta(\vert \zeta\vert^2+\vert\kappa - c\vert^2+\vert\zeta^\prime\vert^2+\vert\kappa^\prime - c^\prime\vert^2-A)d\zeta\right) d\zeta^\prime.\\
\end{split}
\end{equation*}
For ${\bf J}^1_{jp}$, we observe that the sum is only over a uniformly bounded number of $\kappa,\kappa^\prime$ and we can use \eqref{SimpleBd1}. For ${\bf J}^2_{jp}$, we can use \eqref{SimpleBd2} followed by Lemma \ref{VolComp}. Summing over $j,p$, we obtain an acceptable contribution.

\end{proof}

In the proof above, we have use two simple bounds that allow us to cancel two integrals.

 \begin{lemma}\label{simpleLemma}
 
 We have 

\begin{equation}\label{SimpleBd1}
\begin{split}
\sup_{A\in\mathbb{R}}\iint_{\mathbb{R}^2} \delta(\zeta^2+\eta^2-A)d\zeta d\eta& =\pi
\end{split}
\end{equation}
and, for $\mathcal{S}_j$ defined as in \eqref{DefSets} and $R\ge50$,
\begin{equation}\label{SimpleBd2}
\begin{split}
\sup_{A\in\mathbb{R}}\sum_{\vert\kappa\vert\le  R-10}\int_{\mathbb{R}} \mathfrak{1}_{\mathcal{S}_j}(\zeta,\kappa)\delta(\zeta^2-A)d\zeta\lesssim 1,\qquad 0\le j\le R^\frac{1}{2}+1.
\end{split}
\end{equation}  \end{lemma}

\begin{proof}[Proof of Lemma \ref{simpleLemma}]
The first bound is direct after passing to polar coordinates. To prove \eqref{SimpleBd2}, we may assume $R\ge 50$. We first claim that
\begin{equation} \label{kappaId}
\begin{split}
(\zeta,\kappa)\in\mathcal{S}_j,\quad \vert \kappa\vert\le R-10\quad\Rightarrow\quad \vert\zeta\vert\ge R^\frac{1}{2}(R-\vert\kappa\vert-1)^\frac{1}{2}
\end{split}
\end{equation}
Indeed, on $\mathcal{S}_j$, we see that $\zeta^2+(\kappa-c)^2\ge R^2-3\sqrt{R}$ for some $c\in\{0,\frac{1}{2}\}$ and
\begin{equation*}
\begin{split}
\zeta^2\ge (R+\vert \kappa-c\vert)(R-\vert \kappa-c\vert)-3\sqrt{R}\ge R(R-\vert \kappa\vert-1)+R/2-3\sqrt{R}.
\end{split}
\end{equation*}
Eliminating some terms and taking square roots give the result. To prove \eqref{SimpleBd2} we then apply a change of variables along with \eqref{kappaId} to estimate
\begin{equation*}
\begin{split}
\sum_{\vert\kappa\vert\le  R-10}\int \mathfrak{1}_{\mathcal{S}_j}(\zeta,\kappa)\delta(\zeta^2-A)d\zeta
&\lesssim \sum_{\vert\kappa\vert\le R-10}R^{-\frac{1}{2}}\left[ R-\vert\kappa\vert-1\right]^{-\frac{1}{2}}\lesssim 1,
\end{split}
\end{equation*}
which gives \eqref{SimpleBd2}. 
\end{proof}

\section{On volumes of annuli in $\mathbb{R}\times \mathbb{Z}$}

As we saw in the last section, the contribution of the integral $\textbf{I}_{jp}$ is controlled by the following geometric lemma which says that the volume of a (large and thin) annulus in $\mathbb{R}\times\mathbb{Z}$ is proportional to its volume in $\mathbb{R}^2$. The result is essentially Lemma 2.1 from \cite{TaTz}.

\begin{lemma}\label{VolComp}

For $0\le w\le 20\le R$ and $0\le \vert x\vert \le 1/2$,
\begin{equation*}
\begin{split}
V(R,w)=\vert \mathbb{R}_\zeta\times\mathbb{Z}_\kappa\cap \{R^2\le\zeta^2+(\kappa+x)^2\le(R+w)^2\}\vert\lesssim \sqrt{Rw}+Rw.
\end{split}
\end{equation*} As a consequence, for the sets in \eqref{DefSets} we have $\vert\mathcal{S}_j\vert\lesssim 1$ for $0\le j\le R^\frac{1}{2}+1$.
\end{lemma}

\begin{proof}[Proof of Lemma \ref{VolComp}]

Let
\begin{equation*}
\ell(y)=\begin{cases}
\sqrt{(R+w)^2-y^2}&\hbox{ if }R\le \vert y\vert\le R+w\\
\sqrt{(R+w)^2-y^2}-\sqrt{R^2-y^2}&\hbox{ if }0\le \vert y\vert\le R
\end{cases}
\end{equation*}
be the length of the horizontal segment in the annulus under consideration at ordinate $y$. This is maximized at $\vert y\vert=R$ when it is at most $\sqrt{3Rw}$. In addition, for $2^p\le \vert \vert \kappa+x\vert-R\vert\le 2^{p+1}$ and  $32\le 2^p\le R$, we can estimate
\begin{equation*}
\begin{split}
\ell(\kappa+x)&\le\frac{2Rw}{\sqrt{R}\sqrt{R-\kappa-21}}\le 4R^\frac{1}{2}2^{-\frac{p}{2}}w.
\end{split}
\end{equation*}
Summing a bounded number of contributions when $\kappa+x\ge R-50$ and the above bound otherwise, we conclude that the volume under consideration is at most
\begin{equation*}
V\lesssim \sqrt{Rw}+R^\frac{1}{2}w\sum_{p}2^{\frac{p}{2}}\lesssim \sqrt{Rw}+Rw.
\end{equation*}

\end{proof}

\end{document}